\documentclass[11pt]{amsart}
  \RequirePackage{amsmath}
\RequirePackage{amssymb} \RequirePackage{amsthm}
\RequirePackage{epsfig} 
\def\ep{\varepsilon}

\newcommand{\D}{\mathbb{D}}

\newcommand{\Fpf}{\mathcal{F}^{\phi}_p}
\newcommand{\Ftwof}{\mathcal{F}^{\phi}_2}
\newcommand{\Fqf}{\mathcal{F}^{\phi}_q}

\newcommand{\Fif}{\mathcal{F}^{\phi}_\infty}

\newcommand{\N}{\mathbb{N}}

\newcommand{\C}{\mathbb{C}}

\newcommand{\R}{\mathbb{R}}

\newcommand{\vp}{\vp}

\newcommand{\og}{\mathrm{O}}

\def\a{\alpha}               \def\g{\gamma}
\def\d{\delta}            
     \def\om{\omega}      
              
         \def\r{\rho}         \def\z{\zeta}

\DeclareMathOperator{\supp}{supp}

\newtheorem{theorem}{Theorem}
\newtheorem{lemma}[theorem]{Lemma}
\newtheorem{proposition}[theorem]{Proposition}

\newtheorem{corollary}[theorem]{Corollary}

\newtheorem{lettertheorem}{Theorem}

\theoremstyle{definition}
\newtheorem{definition}[theorem]{Definition}

\theoremstyle{remark}
\newtheorem{remark}[theorem]{Remark}
\theoremstyle{remarks}

\newenvironment{Pf}{\noindent{\emph{Proof of}}}{$\hfill\Box$ }

\numberwithin{equation}{section}
\begin{document}
\title[Boundedness of the Bergman projection on an $L^p$ space ]
{Boundedness of the Bergman projection on $L^p$ spaces with exponential weights}

\author{Olivia Constantin}
\address{School of Mathematics, Statistics and Actuarial Science,
University of Kent,
Canterbury, Kent, CT2 7NF,
United Kingdom}
\email{O.A.Constantin@kent.ac.uk}
\address{
Faculty of Mathematics,
University of Vienna,
Nordbergstrasse 15, 1090 Vienna,
Austria}
\email{olivia.constantin@univie.ac.at}

\author{Jos\'e \'Angel Pel\'aez}

\address{Departamento de An´alisis Matem´atico, Universidad de M´alaga, Campus de
Teatinos, 29071 M´alaga, Spain} \email{japelaez@uma.es}

\thanks{The first author was supported in part by the FWF project
P 24986-N25. The second author was supported in part by the Ram\'on y Cajal program
of MICINN (Spain), Ministerio de Edu\-ca\-ci\'on y Ciencia, Spain,
(MTM2011-25502), from La Junta de Andaluc{\'i}a, (FQM210) and
(P09-FQM-4468).}
\date{\today}



\begin{abstract}
Let $v(r)=\exp\left(-\frac{\a}{1-r}\right)$ with $\a>0$, and let  $\D$ be  the unit disc in the complex plane. Denote by $A^p_v$ the subspace of  analytic functions
 of $L^p(\D,v)$ and let $P_v$  be
 the orthogonal projection from
$L^2(\D,v)$ onto $A^2_v$.  In 2004, Dostanic revealed the intriguing fact that $P_v$  is bounded from $L^p(\D,v)$ to $A^p_v$ only for $p=2$,
and he posed the related problem of identifying the duals of  $A^p_v$ for $p\ge 1$, $p\neq 2$.
In this paper we propose a solution to this problem by proving that
that $P_v$ is bounded from $\,L^p(\D,v^{p/2})$ to $A^p_{v^{p/2}}$   whenever $1\le p <\infty$, and, consequently, the dual of
 $A^p_{v^{p/2}}$ for $p\ge 1$ can be identified with $A^{q}_{v^{q/2}}$, where $1/p+1/q=1$.
In addition, we also address a similar question on some classes of weighted Fock spaces.
\end{abstract}
\maketitle


\section{Introduction}
\parindent 1em  Let $\D$ be the unit disc in the complex plane,
$dm(z)=\frac{dx\,dy}{\pi}$ be the normalized area measure on $\D$,
and denote by $H(\D)$ the space of all analytic functions in $\D$.
 A function
$\omega:\D\to (0,\infty)$, integrable over $\D$, is called a
\emph{weight function} or simply a
\emph{weight}. It is \emph{radial} if $\omega(z)=\omega(|z|)$ for all $z\in\D$.
\par For
$0<p<\infty$, the weighted  Bergman space $A^p_w$ is the space of
functions $f\in H(\D)$ such that
\begin{displaymath}
\|f\|_{A^p_w}^p=\int_{\D}|f(z)|^p w(z)\,dm(z)<\infty.
\end{displaymath}
We shall write $L^\infty_w$ for the weighted growth space of measurable functions $f$ such that
$$||f||_{L^\infty_w} =\sup_{z\in\D}|f(z)|w(z)<\infty.$$
We denote $A^\infty_w=H(\D)\cap L^\infty_w$.
 For $w(r)=(\a+1)(1-r^2)^\a$, $\a>-1$, we obtain the standard Bergman spaces $A^p_\a$ \cite{HKZ,Zhu}.
\par Whenever $w$  is a continuous weight, the  $||\cdot||_{A^p_w}$ convergence implies
uniform convergence on compact subsets of $\D$, which in particular gives  that $A^p_w$ is a closed subspace of $L^p(\D,w )$.
In particular, the point
evaluations $L_ z$ (at the point $z\in \D$) are bounded linear
functionals on $A^2_w$. Therefore, there are reproducing kernels
$K_ z\in A^2_w$ with $\|L_ z\|=\|K_ z\|_{A^2_w}$ such that
\begin{displaymath}
L_ z f=f(z)=\langle f, K_ z\rangle =\int_{\D} f(\z)\,\overline{K_
z(\z)}\,w(\z)\,dm(\z), \quad f\in A^2_w.
\end{displaymath}
Since $A^2_w$ is a closed subspace of the Hilbert space $L^2_w$, we may consider the
  orthogonal projection $P_w:L^2(\D,w )\to A^2_w$ that is usually called the
 {\em{ Bergman projection}}. It is precisely  the integral operator
\begin{equation}\label{bp1}
P_w(f)(z)=\int_{\D} f(\z)\,\overline{K_
z(\z)}\,w(\z)\,dm(\z), \quad f\in L^2(\D,w ).
\end{equation}
\par The boundedness of projections on $L^p$-spaces   is an intriguing topic  which has  attracted a lot  attention
in recent years \cite{D1,D3,HKZ,ZeyTams2012,Zhu}.
For the class of standard weights, the Bergman projection
$$P_\a(f)(z)=(\a+1)\int_{\D} f(\z)\frac{(1-|\z|^2)^\a}{(1-z\overline{\z})^{2+\a}}\,dm(\z),$$
 is bounded from $L^p(\D,(1-|z|^2)^\a)$ onto $A^p_\a$ if and only if $1<p<\infty$ \cite{HKZ,Zhu}. This result is
 the key for the description of dual spaces $\left(A^p_\a\right)^\star\cong A^q_\a$, $p>1$\, $\frac1p+\frac1q=1$, and can be  used for obtaining
 Littlewood-Paley formulas in this context \cite[Chapter $4$]{Zhu}. Further results along this line were obtained whenever $\om$ is a non (necessarily) radial
 weight such that $\frac{\om(z)}{(1-|z|)^\eta}$ belongs to some class of Bekoll\'e weights $B_p(\eta)$, $p>1$, $\eta>-1$ \cite{AlCo,BB}.
However,  Dostanic \cite{D1,D3} revealed a behaviour of the Bergman projection in the case of Bergman spaces with the  exponential type weights
$w(r)=(1-r^2)^A\exp\left(\frac{-B}{(1-r^2)^\a}\right)$, $A\in\mathbb{R},\,B,\a>0$,
 that
is in stark contrast with the situation encountered on standard Bergman spaces. More precisely, boundedness only occurs for $p=2$.
This phenomenon has implications for duality issues, as pointed out in \cite{D1}, where the description of the dual of $A^p_w$ for $p\neq 2$, with $w$ as
above, is stated as an open problem.
Dostanic's result \cite{D1} was recently
generalized in \cite{ZeyTams2012}  for a larger class of  weights. This behaviour is reminiscent of that
of the Bergman projection on Fock spaces \cite{tung,ZhuFock}. Other similarities between Fock spaces and Bergman spaces with rapidly decreasing weights
were pointed out in \cite{PP,CP}. Inspired by the Fock space setting, we propose a positive solution for the
boundedness of the Bergman projection for the canonical example of a rapidly decreasing weight $w(z)=e^{-\frac{\alpha}{1-|z|}}$,\, $\a>0$.
Our main result is the following:

\begin{theorem}\label{th:5}
Let $v(r)=\exp\left(-\frac{\a}{1-r}\right)$,\,$\a>0$, and $1\le p <\infty$. Then, the Bergman projection
$$P_v(f)(z)=\int_{\D} f(\z)\,K(z,\zeta)v(z)\,dm(\z)$$
is bounded from $L^p\left(\D,v^{p/2}\right)$ to $A^p_{v^{p/2}}$.
Moreover, $P_v: L^\infty_{v^{1/2}}\to A^\infty_{v^{1/2}}$ is bounded.
\end{theorem}
As a consequence of Theorem \ref{th:5}  we identify  the dual of $A^{p/2}_{v^{p/2}}$ for $p\ge 1$ as $A^{q/2}_{v^{q/2}}$, where $1/p+1/q=1$.
The approach to prove our main result relies on  accurate  estimates for $$||K_r||_{H^1}=M_1(r,K)=\int_0^{2\pi} |K(re^{it})|\, dt,$$ the integral means of the reproducing kernel of $A^2_{v}$, on circles of radius $r<1$ centered at the origin,
 (see Proposition \ref{pr:intmean} below).
 These are obtained using two key tools;  the sharp asymptotic estimates obtained in \cite{Kriete2003} for the moments
of our weight  in terms of the Legendre-Fenchel transform, and an upper estimate of $M_1(r,K)$
 by the $l^1$-norm of the $H^1$-norms of the Hadamard product of $K_r$ with certain smooth polynomials.

Concerning the unboundedness of the Bergman projection,  we highlight  that the decay of the weight plays a role in this problem.  In fact,
whenever the weight is  smooth and decreases rapidly enough, it  enters into the
framework of the result in \cite[Theorem $1.2$]{ZeyTams2012}, so the Bergman projection  is bounded from  $L^p(\D,w )$ to $A^p_w$ only for $p=2$.
\begin{proposition}\label{th:1}
Assume that $w(r)=e^{-2\phi(r)}$ is a radial weight such that \\$\phi:[0,1)\to \R^+$ is a $C^\infty$-function, $\phi'$ is positive on $[0,1)$, $\lim_{r\to 1^-}\phi(r)=\lim_{r\to 1^-}\phi'(r)=+\infty$  and
 \begin{equation}\label{bp4}\lim_{r\to 1^-} \frac{\phi^{ (n)}(r)}{\left(\phi'(r)\right)^n}=0, \quad\text{for any $n\in\N\setminus\{1\}$.
 }\end{equation}
 Then, the Bergman projection  is bounded from  $L^p(\D,w )$ to $L^p(\D,w )$ only for $p=2$.
\end{proposition}
\par In particular, any weight in \cite[section $7$ ]{PP} satisfies the hypotheses of Proposition \ref{th:1}, as well as triple exponential weights of the form
$
\omega(z)=\exp({-e^{e^\frac{1}{1-|z|}}}).
$
\medskip
\par We also investigate some analogous problems in the setting of Fock spaces.
Given $\phi:\C\rightarrow \R^+$ a (nonharmonic) subharmonic function,
we consider
the weighted Fock spaces,
$$\Fpf=\left\{f\in H(\C):\, ||f||^p_{\Fpf}=\int_{\C}|f(z)|^pe^{-p\phi(z)}\,dm(z) <\infty\right\},\quad 0<p<\infty, $$
and
$$\Fif= H(\C)\cap L^\infty_\phi,$$
where $H(\C)$ is the space of entire functions and $ L^\infty_\phi$ is the
 weighted growth space of measurable functions $f$ such that
$$||f||_{L^\infty_\phi} =\sup_{z\in\C}|f(z)|e^{-\phi(z)}<\infty.$$
 We shall say that $\phi$  is radial if $\phi(z)=\phi(|z|)$ for all $z\in\C$.
\par First, we discuss some positive results. The estimates of the Bergman kernel proved in \cite{SeiYouJGA2011} for a large class of Fock spaces induced by radial weights $e^{-2\phi}$,
respectively those obtained in  \cite{marzo-ortega} for Fock spaces ${\Ftwof}$ with nonradial weights, where $\Delta \phi$  is a doubling measure, imply that the Bergman
projection
\begin{equation}\label{bp7}
P_\phi(f)(a)=\int_{\C} f(z)\,\overline{K_
a(z)}\,e^{-2\phi(z)}\,dm(z),
\quad a\in\C,
\end{equation}
is bounded from $L^p(\C, e^{-p\phi})$ to $\Fpf$ for $1\le p< \infty$, and from $L^\infty_\phi$ to $\Fif$.
 From these results we deduce that the dual of ${\Fpf}$ can be identified with ${\Fqf}$, where $1/p+1/q=1$.

\par  Finally, regarding the unboundedness of the Bergman projection on \lq\lq weighted $L^p$ spaces on $\C$\rq\rq,   we note that the
analogue of
Proposition \ref{th:1} works.

\begin{proposition}\label{th:2}
Assume  that $\phi:[0,\infty)\to \R^+$ is a $C^\infty$-function, $\phi'$ is positive on $[0,+\infty)$ and
$\lim_{r\to \infty}\phi(r)=\lim_{r\to \infty}\phi'(r)=+\infty$.
 If
 \begin{equation}\label{bp4f}\lim_{r\to \infty} \frac{\phi^{(n)}(r)}{\left(\phi'(r)\right)^n}=0, \quad\text{for any $n\in\N\setminus\{1\}$,
 }\end{equation}
 then,  Bergman projection  is bounded from  $L^p(\C,e^{-2\phi} )$ to $L^p(\C,e^{-2\phi} )$ only for $p=2$.
\end{proposition}

\par Throughout the paper, the letter $C$ will denote an absolute
constant whose value may change at different occurrences. We also
use the notation $a\lesssim b$ to indicate that there is a
constant $C>0$ with $a\leq C b$, and the notation $a\asymp b$
means that $a\lesssim b$ and $b\lesssim a$.
\section{Weighted Bergman spaces}

We begin with a general result that may be known to specialists, but we include a proof for the sake of completeness.
\begin{lemma}\label{le:est}
Assume that $w=e^{-2\phi}$ is a weight such that its associated kernel function $K(z,\zeta)$ satisfies
\begin{equation}\label{bp10b}
M=\sup_{z\in\D}\int_{\D}\left|K(z,\zeta)\right|\,e^{-\phi(\zeta)-\phi(z)}\,dm(\z)<\infty.
\end{equation}
Then $$P_w: L^p(\D,w^{p/2})\to A^p_{w^{p/2}}$$ is bounded for $1\le p<\infty$. Moreover, Moreover, $$P_w: L^\infty_{w^{\frac{1}{2}}}\to A^\infty_{w^{\frac{1}{2}}}$$ is a  bounded operator.
\end{lemma}
\begin{proof} For simplicity, we write $P$ instead of $P_w$ throughout the proof.
Using  \eqref{bp1} and H\"{o}lder's inequality we deduce
\begin{equation*}\begin{split}
&\left|Pf(z)e^{-\phi(z)} \right|\le\int_{\D}|f(\z)e^{-\phi(\zeta)}| \left|K(z,\zeta)\right|\,e^{-\phi(\zeta)-\phi(z)} \,dm(\z)
\\ &\le \left(\int_{\D}|f(\z)e^{-\phi(\zeta)}|^p \left|K(z,\zeta)\right|\,e^{-\phi(\zeta)-\phi(z)} \,dm(\z)\right)^{\frac{1}{p}}
 \left( \int_{\D}\left|K(z,\zeta)\right|\,e^{-\phi(\zeta)-\phi(z)} \,dm(\z)\right)^{\frac{1}{p'}}
\\ & \le M^{{\frac{1}{p'}}} \left(\int_{\D}|f(\z)e^{-\phi(\zeta)}|^p\left|K(z,\zeta)\right|\,e^{-\phi(\zeta)-\phi(z)} \,dm(\z)\right)^{\frac{1}{p}}.
\end{split}\end{equation*}
So, by the previous inequality, Fubini's theorem and \eqref{bp10b}, we obtain
\begin{equation*}\begin{split}
&\int_{\D}\left|Pf(z)e^{-\phi(z)} \right|^p \,dm(z)
\\ & \le M^{{\frac{p}{p'}}}\int_{\D}\left(\int_{\D}|f(\z)e^{-\phi(\zeta)}|^p\left|K(z,\zeta)\right|\,e^{-\phi(\zeta)-\phi(z)} \,dm(\z)\right) \,dm(z)
\\ & = M^{{\frac{p}{p'}}}\int_{\D}|f(\z)e^{-\phi(\zeta)}|^p\left(\int_{\D}\left|K(z,\zeta)\right|\,e^{-\phi(\zeta)-\phi(z)} \,dm(z)\right) \,dm(\z)
\\ & \le M^{p}_{L^p(\D,w^{p/2})}.
\end{split}\end{equation*}
The inequality
$$||Pf||_{A^\infty_{w^{\frac{1}{2}}}}\le M ||f||_{L^\infty_{w^{\frac{1}{2}}}},$$
is immediate, while a simple application of Fubini's theorem yields
$$||Pf||_{A^1_{w^{\frac{1}{2}}}}\le M ||f||_{L^1\left(\D,w^{\frac{1}{2}}\right)}.$$
This finishes the proof.
\end{proof}
\par Bearing in mind Lemma \ref{le:est}, Theorem \ref{th:5} will been proved once we get the following inequality
\begin{equation}\label{bpbc10}
\sup_{z\in\D}\int_{\D}\left|K(z,\zeta)\right|\,\exp\left(-\frac{\a}{2(1-|z|)}-{\frac{\a}{2(1-|\zeta|)}}\right)\,dm(\z)<\infty,
\end{equation}
where $K$ is the reproducing kernel of the weighted Bergman space $A^2\left(\exp\left(-\frac{\a}{(1-|z|)}\right)\right)$.
It is known that
$$K(z,\zeta)=\sum_{n=0}^\infty \frac{(z\bar{\zeta})^n}{v_{2n+1}},$$
where, for any positive $\lambda$, $v_\lambda=\int_0^1 r^{\lambda}\exp\left(-\frac{\a}{1-r}\right)\,dr$.
Our approach begins with writing  the integral in \eqref{bpbc10} in polar coordinates. So the next result will be a key ingredient in this procedure.
Its proof will be presented in Section \ref{imrk}.
\begin{proposition}\label{pr:intmean}
Let $v(r)=\exp\left(-\frac{\a}{1-r}\right)$,\,\,$\a>0$, and let $K(z)=\sum_{n=0}^\infty \frac{z^n}{v_{2n+1}}$. Then, there is a positive constant
$C$ such that
$$M_1(r,K)\asymp \frac{\exp\left(\frac{\a}{1-\sqrt{r}}\right)}{(1-r)^{\frac32}},\quad r\to 1^-,$$
where
$$M_1(r,K)=\int_0^{2\pi} |K(re^{it})|\, dt, \quad 0<r<1.$$
\end{proposition}
\begin{Pf}{\em{Theorem \ref{th:5}}.}
By Proposition \ref{pr:intmean},
\begin{equation}\begin{split}\label{bp111}
&\int_{\D}\left|K(z,\zeta)\right|\,\exp\left(-\frac{\a}{2(1-|z|)}-{\frac{\a}{2(1-|\zeta|)}}\right)\,dm(\z)
\\ & \asymp \int_0^1 \exp\left(\left(\frac{\a}{1-\sqrt{s|z|}}\right)-\frac{\a}{2(1-|z|)}-{\frac{\a}{2(1-s)}}\right)
\,\frac{sds}{(1-\sqrt{s|z|})^{3/2}}
\end{split}\end{equation}
Put $r=|z|$  and denote the integrand in the last relation above by $f_r(s)$.
Now, a calculation shows that
\begin{equation*}\begin{split}\label{eq:ident2}
\left(\frac{1}{1-\sqrt{sr}}\right)-\frac{1}{2(1-r)}-{\frac{1}{2(1-s)}}
=\frac{1}{2}\left[\left(\frac{\sqrt{s}-\sqrt{r}}{1-\sqrt{sr}}\right)\left(\frac{\sqrt{r}}{1-r}-\frac{\sqrt{s}}{1-s}\right)\right].
\end{split}\end{equation*}
In particular, the above expression is negative.
We assume without loss of generality $r> \frac{1}{2}$.
We provide an upper bound for the last integral in \eqref{bp111} by means of a sum of integrals on four subintervals of $[0,1]$.
Let us first notice that
\begin{equation}\label{p1}
\int_0^{1/2} f_r(s) \,ds\le \frac{1}{2(1-\frac{1}{\sqrt 2})^{3/2}}.
\end{equation}
and that
\begin{equation}\label{p2}
\int_{r-(1-r)^{3/2}}^{r+(1-r)^{3/2}} f_r(s)\, ds\le \frac{1}{(1-\sqrt r)^{3/2}} \int_{r-(1-r)^{3/2}}^{r+(1-r)^{3/2}}\, ds\le C.
\end{equation}
It remains to deal with the integrals $\int_{r+(1-r)^{3/2}}^1 f_r(s)\,ds$ and $\int_{1/2}^{r-(1-r)^{3/2}}f_r(s)\,ds$.
If $N=N(r)$ is the largest positive integer such that
$$
2^{2N}<\frac{1}{1-r},
$$
then
\begin{equation}\label{sumN}
\int_{r+(1-r)^{3/2}}^1 f_r(s)\,ds= \sum_{k=0}^N \int_{r+2^k(1-r)^{3/2}}^{\min(r+2^{k+1}(1-r)^{3/2},1)}f_r(s)\,ds.
\end{equation}
For $r+2^k(1-r)^{3/2}\le s\le\min(r+2^{k+1}(1-r)^{3/2},1)$ we have
\begin{equation}\label{pp1}
\frac{\sqrt{s}-\sqrt{r}}{1-\sqrt{sr}}=\frac{1+\sqrt{sr}}{\sqrt{s}+\sqrt{r}}\frac{s-r}{1-sr}\ge C\frac{s-r}{1-r^2}\ge C\, 2^k (1-r)^{1/2},
\end{equation}
and, by the mean value theorem, there exists $r<x<s$ such that
\begin{equation}\label{pp2}
\frac{\sqrt{s}}{1-s}-\frac{\sqrt{r}}{1-r} =\frac{1+x}{2\sqrt{x}(1-x)^2}(s-r)\ge C \frac{2^k(1-r)^{3/2}}{(1-r)^2}=C\, 2^k (1-r)^{-1/2}.
\end{equation}
Using (\ref{pp1}) and (\ref{pp2}) in (\ref{sumN}) we obtain
\begin{eqnarray}\label{p3}
\int_{r+(1-r)^{3/2}}^1 f_r(s)\,ds&\le& \sum_{k=0}^N \int_{r+2^k(1-r)^{3/2}}^{\min(r+2^{k+1}(1-r)^{3/2},1)} e^{-C4^k}\frac{1}{(1-\sqrt{sr})^{3/2}}\,ds\nonumber\\
&\le& \sum_{k=0}^N \int_{r+2^k(1-r)^{3/2}}^{r+2^{k+1}(1-r)^{3/2}} e^{-C4^k}\frac{1}{(1-r)^{3/2}}\,ds\nonumber\\
&\le&\sum_{k=0}^\infty 2^k e^{-C4^k}<\infty.
\end{eqnarray}

Finally, if $N$ is the biggest positive integer such that
$$
2^N<\frac{r-1/2}{(1-r)^{3/2}},
$$
we have
\begin{equation}\label{p21}
\int_{1/2}^{r-(1-r)^{3/2}} f_r(s)\,ds\le\sum_{k=0}^N \int_{r-2^{k+1}(1-r)^{3/2}}^{r-2^k(1-r)^{3/2}}f_r(s)\,ds.
\end{equation}
Consider now $s\in[r-2^{k+1}(1-r)^{3/2},r-2^k(1-r)^{3/2}]$. Then
\begin{eqnarray}\label{pp21}
\frac{\sqrt{r}-\sqrt{s}}{1-\sqrt{sr}}&\ge& C \frac{r-s}{1-sr}\ge C \frac{2^k(1-r)^{3/2}}{1-r(r-2^{k+1}(1-r)^{3/2})}\\
&=& C \frac{2^k(1-r)^{3/2}}{(1-r)(1+r+r2^{k+1}(1-r)^{1/2})}\nonumber\\
&\ge& C \frac{2^k(1-r)^{1/2}}{2+2^{k+1}(1-r)^{1/2}}\ge C \frac{2^k(1-r)^{1/2}}{1+2^{k}(1-r)^{1/2}}.\nonumber
\end{eqnarray}
Moreover, we use the fact that the function $x\mapsto \frac{\sqrt{x}}{1-x}$ is increasing for $x>0$ to deduce
\begin{eqnarray}\label{pp22}
\frac{\sqrt{r}}{1-r}-\frac{\sqrt{s}}{1-s}&\ge& \frac{\sqrt{r}}{1-r} -\frac{\sqrt{r-2^k(1-r)^{3/2}}}{1-r+2^k(1-r)^{3/2}}\\
&=&\frac{\sqrt{r}(1+2^k(1-r)^{1/2})-\sqrt{r-2^k(1-r)^{3/2}} }{(1-r)(1+2^k(1-r)^{1/2})}\nonumber\\
&\ge& \frac{2^{k-1}(1-r)^{-1/2}}{1+2^k(1-r)^{1/2}}, \nonumber
\end{eqnarray}
since
\begin{eqnarray*}
\sqrt{r}(1+2^k(1-r)^{1/2})-\sqrt{r-2^k(1-r)^{3/2}}&=&\frac{(1-r)^{1/2} [r 2^{k+1}+r2^{2k}(1-r)^{1/2}+2^k(1-r)]}{\sqrt r(1+2^k(1-r)^{1/2})+\sqrt{r-2^k(1-r)^{3/2}}}\\
&\ge& (1-r)^{1/2}\frac{r 2^{k+1}+r2^{2k}(1-r)^{1/2}}{2+2^k(1-r)^{1/2}}\\
&\ge& 2^{k-1}(1-r)^{1/2}.
\end{eqnarray*}
Combining  relations (\ref{p21})-(\ref{pp22}) we obtain
\begin{eqnarray}\label{li3}
&&\int_{1/2}^{r-(1-r)^{3/2}} f_r(s)\,ds\\
&\le& 2^{3/2}\sum_{k=0}^N \exp\left(-C \frac{2^{2k}} {(1+2^k(1-r)^{1/2})^2}\right) \int_{r-2^{k+1}(1-r)^{3/2}}^{r-2^k(1-r)^{3/2}} \frac{1}{(1-sr)^{3/2}}\,ds\nonumber\\
&\le & \sum_{k=0}^N \exp\left(-C \frac{2^{2k}} {(1+2^k(1-r)^{1/2})^2}\right)  \frac{2^{k+3/2}}{(1+2^k(1-r)^{1/2})^{3/2}}\nonumber\\
&\le & 2^{5/2} \int_{\frac{1}{2}}^{\frac{1}{(1-r)^{3/2}}} e^{-\frac{Cx^2}{(1+(1-r)^{1/2}x)^2}} \frac{1}{(1+(1-r)^{1/2}x)^{3/2}}\, dx,\nonumber
\end{eqnarray}
where the last step follows in view of the inequalities
$$
\frac{2^k}{(1+2^k(1-r)^{1/2})^{3/2}}\le 2\int_{2^{k-1}}^{2^k}\frac{dx}{(1+(1-r)^{1/2}x)^{3/2}}
$$
and
$$
\frac{2^{2k}} {(1+2^k(1-r)^{1/2})^2}\ge \frac{2^{2k}}{4(1+2^{k-1}(1-r)^{1/2})^2}\ge \frac{x^2}{4(1+(1-r)^{1/2}x)^2}, \quad 2^{k-1}\le x\le 2^k.
$$
Now
\begin{equation}\begin{split}\label{li1}
&\int_{1/2}^{(1-r)^{-1/2}} \exp\left(-\frac{Cx^2}{(1+(1-r)^{1/2}x)^2}\right) \frac{dx}{(1+(1-r)^{1/2}x)^{3/2}}
\\ & \le \int_{1/2}^{(1-r)^{-1/2}} e^{-\frac{Cx^2}{4}}\,dx\le \int_{1/2}^\infty  e^{-\frac{Cx^2}{4}}\,dx <\infty.
\end{split}\end{equation}
On the other hand, performing the substitution $t=(1-r)^{1/2}x$ we obtain
\begin{eqnarray}\begin{split}\label{li2}
&\int_{(1-r)^{-1/2}}^{(1-r)^{-3/2}} \exp\left(-\frac{Cx^2}{(1+(1-r)^{1/2}x)^2}\right) \frac{dx}{ (1+(1-r)^{1/2}x)^{3/2} }
\\ &=\int_1^{1/(1-r)} \exp\left({\frac{-Ct^2}{(1-r)(1+t)^2}}\right) \frac{dt}{\sqrt{1-r}(1+t)^{3/2}}\nonumber
\\ &\le (1-r)^{-1/2} e^{-\frac{C}{4(1-r)} } \int_1^{1/(1-r)} \frac{dt}{(1+t)^{3/2}}\rightarrow 0,
\end{split}\end{eqnarray}
as $r\rightarrow1$.
Taking into account relations (\ref{li1})-(\ref{li2}) we now return to (\ref{li3}) to deduce
\begin{equation}
\sup_{1/2<r<1}\int_{1/2}^{r-(1-r)^{3/2}} f_r(s)\,ds<\infty,
\end{equation}
and, with this, the proof is complete.

\end{Pf}

\begin{remark}\label{sublinear}
It follows from the proof of Theorem \ref{th:5} that the sublinear operator
$$
\tilde P_v(f)(z)=\int_{\C} |f(\z)|\,\left|K(z,\zeta)\right| v(z)\,dm(\z)
$$
is bounded from $L^p\left(\D, v^{p/2}\right)$ to $A^p_{v^{p/2}}$ for $p\ge 1$, respectively from $L_{v^{1/2}}^\infty$ to $A_{v^{1/2}}^\infty$.
\end{remark}

\subsection{Integral means of the reproducing kernel. }\label{imrk}
 We  need some preliminary results.

\begin{lemma}\label{moments}
For each $\a,\lambda \in (0,\infty)$ and $n\in\N\cup\{0\}$, let us consider
\begin{equation*}\begin{split}
 v_{\lambda,\log^n} &=\int_0^1 r^{\lambda}\left(\log\frac{1}{r}\right)^n\exp\left(-\frac{\a}{1-r}\right)\,dr.
\end{split}\end{equation*}
Then,
 \begin{equation*}\begin{split}
 v_{\lambda,\log^n} &\asymp \lambda^{-\frac{2n+3}{4}}\exp\left(-2\sqrt{\a \lambda}\right),\quad \lambda \to\infty,
\end{split}\end{equation*}
where the constants involved may depend on $\a$ and $n$.
\end{lemma}
\begin{proof}
Let us observe that
 \begin{equation}\label{comp}
 \exp\left(-\frac{\a}{1-r}\right)\le \exp\left(-\frac{\a}{\log\frac{1}{r}}\right)\le \exp\left(\a-\frac{\a}{1-r}\right) ,\quad 0<r<1,
 \end{equation}
  so it is enough to estimate
$$\tilde{w}_{\lambda,\log^n}=\int_0^1 r^{\lambda}\left(\log\frac{1}{r}\right)^n\exp\left(-\frac{\a}{\log\frac{1}{r}}\right)\,dr.$$
A change of variables $\log\frac{1}{r}=t$,
$$\tilde{w}_\lambda=\int_0^\infty e^{-\frac{\a}{t}-(\lambda+1) t+n\log t}\,dt.$$
 Then, take  $v(t)=\frac{\a}{t}-n\log t$ and consider its Legendre-Fenchel transform
\begin{equation}\label{legendre}
L(x)=\inf_{0<t<\infty}\left[v(t)+xt\right].
\end{equation}
A simple calculation shows that $L(x)=\sqrt{n^2+4\a x}+n\log\frac{2x}{\sqrt{n^2+4\a x}+n}$, so
$$e^{-L(x)}\asymp x^{-\frac{n}{2}}\exp\left(-2\sqrt{\a x}\right),\quad x\to\infty.$$
Moreover,
$$-L''(x)\asymp
x^{-\frac{3}{2}},\quad x \to\infty.$$
 So by \cite[Theorem $1$]{Kriete2003} we obtain the estimate for  $v_{\lambda,\log^n}$. This finishes the proof.
\end{proof}
Now let us consider two sequence of polynomials  with useful properties with respect to smooth partial sums.
\par Firstly, given a  $C^\infty$-function $\Phi:\mathbb{R}\to\C$  with compact
support $\supp(\Phi)$, we set
    $$
    A_{\Phi,m}=\max_{s\in\mathbb{R}}|\Phi(s)|+\max_{s\in\mathbb{R}}|\Phi^{(m)}(s)|,
    $$
and we consider the polynomials
    $$
    W_n^\Phi(e^{i\theta})=\sum_{k\in\mathbb
    Z}\Phi\left(\frac{k}{n}\right)e^{ik\theta},\quad n\in\N.
    $$
    For any function $f(z)=\sum_{k=0}^\infty a_k z^k\in H(\D)$, we write $$(W_n^\Phi\ast f)(z)=\sum_{k=0}^\infty \Phi\left(\frac{k}{n}\right) a_kz^k$$
for the Hadamard product of $W_n^\Phi$ and $f$.
With this notation we can state the next result, which follows from the considerations in
\cite[p.$111$-$113$]{Pabook}. We denote by $H^p$ the classical Hardy spaces on $\D$ (see \cite{Duren1970}).

\begin{lettertheorem}\label{th:cesaro}
Let $\Phi:\mathbb{R}\to\C$ be a $C^\infty$-function with compact
support $\supp(\Phi)$.  The following
assertions hold:
\begin{itemize}
\item[\rm(i)] There exists a constant $C>0$ such that
    $$
    \left|W_n^\Phi(e^{i\theta})\right|\le C\min\left\{
    n\max_{s\in\mathbb{R}}|\Phi(s)|,
    n^{1-m}|\theta|^{-m}\max_{s\in\mathbb{R}}|\Phi^{(m)}(s)|
    \right\},
    $$
for all $m\in\N\cup\{0\}$, $n\in\N$ and $0<|\theta|<\pi$.
\item[\rm(ii)] If $0<p\le 1$ and $m\in\N$ with $mp>1$, there exists a constant $C(p)>0$ such that
    $$
    \left(\sup_{n}\left|(W_n^\Phi\ast f)(e^{i\theta})\right|\right)^p\le CA_{\Phi,m}
    M(|f|^p)(e^{i\theta})
    $$
for all $f\in H^p$. Here $M$ denotes the Hardy-Littlewood
maximal-operator
    $$
    M(|f|)(e^{i\theta})=\sup_{0<h<\pi}\frac{1}{2h}\int_{\theta-h}^{\theta+h}|f(e^{it})|\,dt.
    $$
\item[\rm(iii)]
 Then, for each $p\in(0,\infty)$ and $m\in\N$ with $mp>1$, there exists a constant
$C=C(p)>0$ such that
    $$
    \|W_n^\Phi\ast f\|_{H^p}\le C A_{\Phi,m}\|f\|_{H^p}
    $$
for all $f\in H^p$ and $n\in\N$.
\end{itemize}
\end{lettertheorem}

Secondly, let us consider a useful particular case of the previous construction. We define  the sequence of polynomials
 $\{V_n\}_{n=0}^\infty$  \cite[Section $2$]{JevPac98} (see also \cite[p. 9]{Padec}) as follows.
Let $\Psi$ be a $C^\infty$-function on $\mathbb{R}$ such that
\begin{enumerate}
\item $\Psi(t)=1$ for $t\le 1$,
\item $\Psi(t)=0$ for $t\ge 2$,
\item $\Psi$ is decreasing and positive on the interval $(1,2)$.
\end{enumerate}
Set $\psi(t)=\Psi\left(\frac{t}{2}\right)-\Psi(t)$. Let $V_0(z)=1+z$ and for $n\ge 1$, let
\begin{equation}\label{vn}
V_n(z)=W^\psi_{2^{n-1}}(z)=\sum_{k=0}^\infty \psi\left(\frac{k}{2^{n-1}}\right)z^k=\sum_{k=2^{n-1}}^{2^{n+1}-1} \psi\left(\frac{k}{2^{n-1}}\right)z^k.
\end{equation}
These polynomials have the following properties with regard to smooth partial sums \cite{JevPac98} (see also \cite[p.$111$-$113$]{Pabook}):
\begin{equation}\begin{split}\label{propervn}
&f(z)=\sum_{n=0}^\infty (V_n\ast f)(z),\quad\text{for any $f\in H(\D)$},
\\ & ||V_n\ast f||_{H^p}\le C||f||_{H^p},\quad\text{for any $f\in H^p$ and $0<p<\infty$},
\\ & ||V_n||_{H^p}\asymp 2^{n(1-1/p)}, \quad\text{for any  $0<p<\infty$}.
\end{split}\end{equation}

\par With these preparations  we are ready to prove the estimate for the integral means of the kernel.

\begin{Pf}{\em{Proposition \ref{pr:intmean}.}}

First, we shall prove the upper estimate
\begin{equation}\label{upper}
M_1(r,K)\lesssim \frac{\exp\left(\frac{\a}{1-\rho}\right)}{1-\rho},\quad \rho=\sqrt{r}.
\end{equation}
\par {\bf{Step $\mathbf{1}$.}}
 It follows from \eqref{propervn} that
\begin{equation}\label{fbnorm}
M_1(r,K)=||f||_{H^1}\le \sum_{n=0}^\infty ||V_n\ast f||_{H^1},
\end{equation}
where $f(z)=\sum_{n=0}^\infty \frac{r^n}{v_{2n+1}}z^n.$
 Next each $n=3,4,\dots$, let us consider
$$F_n(x)=\frac{r^x}{v_{2x+1}},
\quad  x\in[2^{n-1},2^{n+1}].$$
Applying Lemma \ref{moments} for $n=0,1,2$ we have that
$$\frac{\int_0^1s^{2x+1} \log\frac{1}{s}v(s)\,ds}{v_{2x+1}}\asymp x^{-1/2},\quad\text{and}\quad \frac{\int_0^1s^{2x+1} \left(\log\frac{1}{s}\right)^2v(s)\,ds}{v_{2x+1}}\asymp x^{-1},\quad x\to\infty$$
 so there is a positive constant $C$ such that
\begin{equation*}\begin{split}
    A_{F_n,2}= &\max_{x\in[2^{n-1},2^{n+1}]}|F_n(x)|+\max_{x\in[2^{n-1},2^{n+1}]}|F_n''(x)|
    \\ & \le C \max_{x\in[2^{n-1},2^{n+1}]}F_n(x)=CM_n
    ,\quad n\in\N,
\end{split}\end{equation*}
 where    $M_n=\max_{x\in[2^{n-1},2^{n+1}]}F_n(x)$.
\medskip\par
For each $n=3,4,\dots$,  choose $\Phi_n$ a $C^\infty$-function with compact support contained in $[2^{n-2},2^{n+2}]$ such that
$\Phi_n(x)=F_n(x)=\frac{r^x}{v_{2x+1}}$ if     $x\in[2^{n-1},2^{n+1}]$ and
 \begin{equation}\label{phin}
    A_{\Phi_n,2}=\max_{x\in\mathbb{R}}|\Phi_n(x)|+\max_{x\in\mathbb{R}}|\Phi''_n(x)|\le C M_n,\quad n=3,4,\dots
    \end{equation}
Since
 $$W_1^{\Phi_n}(e^{i\theta})=\sum_{k\in\mathbb
    Z}\Phi_n\left(k\right)e^{ik\theta}= \sum_{k\in\mathbb
    Z\cap[2^{n-2},2^{n+2}]}\Phi_n\left(k\right)e^{ik\theta} $$
    and  by \eqref{vn} $supp\widehat{V_n}\subset [2^{n-1}, 2^{n+1}-1]\subsetneq [2^{n-2},2^{n+2}]$, then
       \begin{equation*}\begin{split}
     \left(V_n \ast f\right)(z) &=\sum_{k=2^{n-1}}^{2^{n+1}-1} \psi\left(\frac{k}{2^{n-1}}\right)\frac{r^k}{v_{2k+1}}z^k
  \\ &  =\sum_{k=2^{n-1}}^{2^{n+1}-1} \psi\left(\frac{k}{2^{n-1}}\right)\Phi_n(k)z^k
  \\ &=\left(W_1^{\Phi_n}\ast V_n\right)(z)
    \end{split}\end{equation*}
  This together with Theorem \ref{th:cesaro},  \eqref{phin} and \eqref{propervn}, implies that
    \begin{equation*}\begin{split}
    ||V_n \ast f||_{H^1}&=
    ||W_1^{\Phi_n}\ast V_n||_{H^1}
   \\ & \le C A_{\Phi_n,2}||V_n||_{H^1}
  \\ & \le CM_n ||V_n||_{H^1} \le CM_n \quad n=3,4\dots.
\end{split}\end{equation*}
So, bearing in mind \eqref{fbnorm}
\begin{equation}\label{endstep1}
M_1(r,K)\le C+\sum_{n=3}^\infty M_n.
\end{equation}
\medskip
\par {\bf{Step $\mathbf{2}$.}}
Recall from  Lemma \ref{moments}  that
\begin{equation}\label{comp}
\frac{r^x}{v_{2x+1}}=\frac{\rho^{2x}}{v_{2x+1}}\asymp   \rho^{2x} (2x)^{\frac{3}{4}}e^{2\sqrt{2\a x}}=e^{h(2x)},
\end{equation}
where  $h(x)=\frac{3}{4}\log(x)+2\sqrt{\a x}+x\log \rho$, $x\in (0,\infty)$.
A calculation shows that $h$ increases in $(0,x_\rho)$ and
decreases in $(x_\rho,\infty)$, where $x_\rho=\left(\frac{\sqrt{\a+3\log\frac{1}{\rho}}+\sqrt{\a}}{2\log\frac{1}{\rho}} \right)^2$. Moreover,
\begin{equation}\label{max}
\sup_{x\in (0,\infty)} e^{h(x)}=e^{h(x_\rho)}\asymp \left(\frac{1}{\log\frac{1}{\rho}}\right)^{3/2}\exp\left(\frac{\a}{\log\frac{1}{\rho}}\right)
\asymp \left(\frac{1}{1-\rho}\right)^{3/2}\exp\left(\frac{\a}{1-\rho}\right).
\end{equation}
\medskip \par {\bf{Step $\mathbf{3}$.}}
Choose $n_0\in \N$ such that $2^{n_0}\le x_\rho<2^{n_0+1}$ and split the sum in \eqref{endstep1} as follows
\begin{equation}\label{step3}
\sum_{n=3}^\infty M_n = \sum_{n=3}^{n_0-3} M_n+ \sum_{n=n_0-2}^{n_0+1} M_n +\sum_{n=n_0+2}^\infty M_n
\end{equation}
Bearing in mind \eqref{comp} and the monotonicity of $h$, we deduce  for $3\le n\le n_0-3$
\begin{equation}\label{step31}
M_n=\max_{x\in[2^{n-1}, 2^{n+1}]}\frac{\rho^{2x}}{v_{2x+1}}\asymp  \max_{x\in[2^{n-1}, 2^{n+1}]} e^{h(2x)}=e^{h(2^{n+2})}.
\end{equation}

It follows from \eqref{comp} and   \eqref{max} that
\begin{equation}\label{step32}
\sum_{n=n_0-2}^{n_0+1} M_n\lesssim \left(\frac{1}{1-\rho}\right)^{3/2}\exp\left(\frac{\a}{1-\rho}\right).
\end{equation}
Using again \eqref{comp} and the monotonicity of $h$, whenever $ n\ge n_0+2$
\begin{equation*}
M_n=\max_{x\in[2^{n-1},2^{n+1}]}\frac{\rho^{2x}}{v_{2x+1}}\asymp  \max_{x\in[2^{n-1}, 2^{n+1}]} e^{h(2x)}=e^{h(2^{n})},
\end{equation*}
which together with \eqref{step3}, \eqref{step31} and \eqref{step32} gives that
\begin{equation}\label{step33}
\sum_{n=3}^\infty M_n \lesssim \sum_{n=4}^{n_0-2} e^{h(2^{n+1})}+  \sum_{n=n_0+2}^\infty e^{h(2^{n})} + \left(\frac{1}{1-\rho}\right)^{3/2}\exp\left(\frac{\a}{1-\rho}\right).
\end{equation}
Next,  if $n\le n_0-2$,
$$ e^{h(2^{n+1})}\le \frac{1}{2^{n+1}}\sum_{k=2^{n}}^{2^{n+1}-1}e^{h(2k)}\le 4  \sum_{k=2^{n}}^{2^{n+1}-1}\frac{e^{h(2k)}}{k+1}$$
and for $n\ge n_0+2$
$$ e^{h(2^{n})}\le \frac{1}{2^{n-2}}\sum_{k=2^{n-2}}^{2^{n-1}-1}e^{h(2k)}\le 4  \sum_{k=2^{n-2}}^{2^{n-1}-1}\frac{e^{h(2k)}}{k+1}.$$
Joining the above two inequalities with \eqref{step33} and \eqref{endstep1}, we get
\begin{equation}\begin{split}\label{im2}
M_1(r,K) & \lesssim \left(\frac{1}{1-\rho}\right)^{3/2}\exp\left(\frac{\a}{1-\rho}\right)+\sum_{k=1}^\infty\frac{e^{h(2k)}}{k+1}
\\ & \asymp \left(\frac{1}{1-\rho}\right)^{3/2}\exp\left(\frac{\a}{1-\rho}\right)+\sum_{k=1}^\infty(2k)^{-\frac{1}{4}}\exp\left(2\sqrt{2\a k}\right) \rho^{2k}
\\ & \le \left(\frac{1}{1-\rho}\right)^{3/2}\exp\left(\frac{\a}{1-\rho}\right)+ \sum_{n=2}^\infty
n^{-\frac{1}{4}}\exp\left(2\sqrt{\a n}\right) \rho^{n}.
\end{split}\end{equation}
 Next, let us consider $N(x)=2\sqrt{\a x}-\frac{1}{4}\log(x)$. A new calculation shows that its inverse Legendre-Fenchel transform (see \cite{Kriete2003}) is
 $$u(t)=\frac{2\a+t+2\sqrt{\a(\a-t)}}{4t}-\frac{1}{2}\log\left(\frac{\sqrt{\a}+\sqrt{\a-t}}{2t}\right),$$
 so it follows that
 $$e^{u(t)}\asymp \sqrt{t}e^{\frac{\a}{t}},\quad t\to 0^+,$$
 and
 $$u''(t)\asymp \frac{1}{t^3},\quad t\to 0^+.$$
 So, by \cite[Corollary $1$]{Kriete2003}
 $$\sum_{n=2}^\infty
n^{-\frac{1}{4}}\exp\left(2\sqrt{\a n}\right) \rho^{n}\asymp \frac{\exp\left(\frac{\a}{1-\rho}\right)}{1-\rho},$$
which together with \eqref{im2} finishes the proof of the upper estimate for $M_1(r,K)$.
\par On the other hand, choose $N_0\in \N$ such that $N_0\le \frac{x_\rho}{2}<N_0+1$. Then, by Lemma \ref{moments}
$$ \frac{r^N_0}{v_{2N0+1}}\asymp   \rho^{2N0} (2N0)^{\frac{3}{4}}e^{2\sqrt{2\a N0}}=e^{h(2N0)}\asymp e^{h(x_\rho)},$$
which together with the well known estimate on coefficients of $H^1$-functions \cite[Theorem $6.4$]{Duren1970}, implies that
$$M_1(r,K)\ge \frac{r^N_0}{v_{2N0+1}}\asymp \left(\frac{1}{1-\rho}\right)^{3/2}\exp\left(\frac{\a}{1-\rho}\right).$$
This concludes the proof.
\end{Pf}

\subsection{Dual spaces}
As as immediate consequence of Theorem \ref{th:5} we have

\begin{corollary}\label{dual} Assume $v(r)=\exp{\left(-\frac{\a}{1-r}\right)}$,\,\,$\a>0$. For $p>1$, the dual of  $A^p_{v^{p/2}}$ is $A^q_{v^{q/2}}$,
under the pairing
\begin{equation}\label{pairing}
\langle f,g \rangle =\int_\D f \bar g \, vdm,
\end{equation}
where $1/p+1/q=1$. Moreover, the dual of $A^1_{v^{1/2}}$ is $A^\infty_{v^{1/2}}$ under the pairing (\ref{pairing}).
\end{corollary}

\begin{proof}
The fact that any function in $A^q_{v^{q/2}}$ induces a bounded linear functional on $A^p_{v^{p/2}}$ with the pairing (\ref{pairing})
follows by a simple application of H\"older's inequality  in the case $p>1$, while the case $p=1$ is obvious.

Conversely, assume that $L$ is a bounded linear functional on $A^p_{v^{p/2}}$. By the Hahn-Banach theorem we can extend $L$
to a functional $\tilde L$ on $L^p(\D, v^{p/2})$ with $\|L\|=\|\tilde L\|$.  Hence there exists a function $h\in L^q(v^{q/2})$ for $p>1$, respectively $h\in L^\infty_{v^{1/2}}$
for $p=1$, such that
$$
L f=\int_\D  f \bar h \, vdm, \quad f\in A^p_{v^{p/2}},
$$
and $\|\tilde L\|=\|h\|_{L^q(v^{q/2})}$.
Then, in view of Remark \ref{sublinear}, we can use Fubini's theorem to deduce
\begin{eqnarray*}
L f&=& \int_\D  f \bar h \, vdm=\int_\D  P_v f \, \bar h \, vdm \\
&=& \int_\D \left(\int_\D f(\z) K(z,\z) v(\z)\, dm(\z) \right) \bar h(z) \, v(z)dm(z)\\
&=& \int_\D  f(\z)   \overline{P_v h (\z)} \, v(\z)dm(\z)\,.
\end{eqnarray*}
By Theorem \ref{th:5} we obtain $P_v h \in A^q_{v^{q/2}}$  and
$$
\|P_v h\|_{A^q_{v^{q/2}}}\lesssim \|h\|_{L^q(v^{q/2})}=\|L\|,\quad \text{if $p>1$.}
$$
Moreover,  for $p=1$  we have $P_v h\in A^\infty_{v^{1/2}}$ and $$
\|P_v h\|_{A^\infty_{v^{1/2}}}\lesssim \|h\|_{L^\infty(v^{1/2})}=\|L\|,
$$
and the proof is complete.
\end{proof}

\subsection{$L^p$-unboundedness}

\par  This section is mainly devoted to prove  Proposition \ref{th:1}.  For this aim, we shall use \cite[Theorem $1.2$]{ZeyTams2012}.
\begin{lettertheorem}\label{th:zey}
If $w$ is a radial weight on $\D$ which satisfies
\par $(i)$\, $w$ is  $C^\infty([0,1))$,
\par $(ii)$\,$\lim_{r\to 1^-} w^{(n)}(r)=0$ for any $n\in\N$,
\par $(iii)$\, for any $n\in\N$ there exists $a_n\in (0,1)$ such that $(-1)^nw^{(n)}$ is non-negative on the interval
$(a_n,1)$.
\par Then, if $0<p<\infty$,
the Bergman projection  is bounded from  $L^p(\D,w )$ to $L^p(\D,w )$ only for $p=2$.
\end{lettertheorem}

\begin{definition}
Given $\phi\in C^\infty([0,1))$, we shall say that $Q_n(\phi)$ is a product of level $n$ (for $\phi$) if
 $Q_n(\phi)=\Pi_{j=1}^n\left(\phi^{(j)}\right)^{m(j)}$ where $m(j)\in \N\cup\{0\}$ and
 $\Sigma_{j^=1}^n\, jm(j)=n.$
\end{definition}
\par We  shall also use the next lemma which proof will be omitted.
\begin{lemma}\label{le:bp1}
Assume that $\phi\in C^\infty([0,1))$. Then
\begin{enumerate}
\item  If  $Q_n(\phi)=\Pi_{j=1}^n\left(\phi^{(j)}\right)^{m(j)}$ is is a product of level $n$, then $\phi'Q_n(\phi)=(\phi')^{m(1)+1}\Pi_{j=2}^n\left(\phi^{(j)}\right)^{m(j)}$ is a product of  level $(n+1)$.
\item The derivative of a product of level $n$ is a finite linear combination of products of level $(n+1)$.
\end{enumerate}
\end{lemma}

\begin{Pf}{\em{of Proposition \ref{th:1}}.}
\par First we observe that
 \begin{equation}\label{bp3}
 \lim_{r\to 1^-} \left(\phi'(r)\right)^n e^{-2\phi(r)}=0,\quad\text{ for any $n\in\N$}.\end{equation}
This holds if and only if
$$\lim_{r\to 1^-}2\phi(r)-n\log\phi'(r)=+\infty.$$
So it suffices to prove
$$\lim_{r\to 1^-}\frac{\phi(r)}{\log\phi'(r)}=+\infty,$$
but this follows from L'Hospital's rule and \eqref{bp4} for $n=2$.

\par Next,
we are going to see that $w=e^{-2\phi}$ satisfies the hypotheses of Theorem \ref{th:zey}. Since $\phi\in C^\infty([0,1))$,  Theorem \ref{th:zey} (i) is clear.
Observe that $w'=-2\phi'   e^{-2\phi}$ and  $w''=[4(\phi')^2-2\phi'']   e^{-2\phi}$. Indeed reasoning by induction, it follows from Lemma \ref{le:bp1} that
\begin{equation}\label{bp5}
w^{(n)}=P_n(\phi)e^{-2\phi},\quad\text{where $P_n(\phi)=(-1)^n2^n\left(\phi'\right)^n+ R_n(\phi)$,}
\end{equation}
where the harmless term $R_n(\phi)$ is a finite linear combination of products of level $n$ such that the exponent of $\phi'$ is at most $(n-1)$.
So, by \eqref{bp4} and  (\ref{bp5})
\begin{equation}\label{bp6}
 \lim_{r\to 1^-} \frac{P_n(\phi)(r)}{\left(\phi'(r)\right)^n}=(-1)^n2^n,
 \end{equation}
 which together  (\ref{bp3}) gives
 $$\lim_{r\to 1^-} w^{(n)}(r)=\lim_{r\to 1^-}\frac{P_n(\phi)(r)}{\left(\phi'(r)\right)^n}\left(\phi'(r)\right)^ne^{-2\phi(r)}=0,$$
 that is, Theorem \ref{th:zey} (ii) holds.
 \par On the other hand, by \eqref{bp5},
 $$(-1)^nw^{(n)}=\left[2^n\left(\phi'\right)^n+ (-1)^nR_n(\phi)\right] e^{-2\phi}=\left[2^n+ \frac{(-1)^nR_n(\phi)}{\left(\phi'\right)^n}\right]\left(\phi'\right)^n e^{-2\phi}.$$
Since $\phi'$ is positive on $[0,1)$ and $\lim_{r\to 1^-} \frac{(-1)^nR_n(\phi)}{\left(\phi'(r)\right)^n}=0$, we deduce that there is
 $a_n(\phi)\in (0,1)$ such that $(-1)^nw^{(n)}(r)>0$  on the interval
$(a_n,1)$. This fact together with Theorem \ref{th:zey} finishes the proof.
\end{Pf}

\section{Weighted Fock spaces}
We begin by introducing a couple of definitions.
\begin{definition}
We shall say that  $\phi:\C\to \R^+$ belongs to the class $\mathcal{D}$ if it is a subharmonic function (not necessarily radial) having the property that
$\mu=\Delta\phi$ is a doubling measure.
\end{definition}
\begin{definition}
We shall say that  $\phi:\C\to \R^+$ belongs to the class $\mathcal{S}$ if it is radial and the function $\Psi(x)=2\phi(\sqrt{x})$  satisfies that
$$\Psi'(x)>0,\quad\Psi''(x)\ge 0\quad \Psi'''(x)\ge 0$$
and there exists a real number $\eta<\frac{1}{2}$ such that
$$2\Psi''(x)+x\Psi'''(x)=\og\left(\frac{\left(\Psi'(x)+x\Psi''(x)\right)^{1+\eta}}{\sqrt{x}}\right),\quad x\to\infty.$$
\end{definition}
\par It is worth to notice that $\mathcal{D}\cap\mathcal{S}$ contains functions,  as shown by the examples $\phi(z)=|z|^m$ with $2<m<4$.
However, roughly speaking, if  $\phi$ is radial, smooth enough an increases very rapidly, then $\phi\in \mathcal{S}\setminus\mathcal{D}$.
\begin{theorem}\label{th:3}
Asumme that $1\le p<\infty$ and $\phi:\C\to \R^+$ belongs to $\mathcal{D}\cup\mathcal{S}$. Then, the Bergman projection
$$P_\phi(f)(z)=\int_{\C} f(\z)\,K(z,\zeta)\,e^{-2\phi(\z)}\,dm(\z)$$
is bounded from $L^p(\C, e^{-p\phi})$ to $\Fpf$. Moreover, $P_\phi: L^\infty_\phi\to \Fif$ is bounded.
\end{theorem}

\begin{proof}
We shall split the proof in two cases.
\\{\bf{Case $\mathbf{1}$.}} Assume that $\phi\in\mathcal{D}$.
 We shall follow ideas and the notation from \cite{marzo-ortega}.
Adapting the approach used in the proof of Lemma \ref{le:est} to the complex plane setting, we see that the result will follow once we can show that there is a positive constant $M$ such that
\begin{equation}\label{bp10}
\sup_{z\in\C}\int_{\C}\left|K(z,\zeta)\right|\,e^{-\phi(\zeta)-\phi(z)}\,dm(\z)\le M.
\end{equation}
In order to prove \eqref{bp10} we need to introduce a technical tool.
We denote by  $\rho(z)$ the positive radius for which we have $\mu(D(z,\rho(z)))=1,\, z\in\C$.
The function $\rho^{-2}$ can be regarded as a regularized version of $\Delta \phi$ (see \cite{{MarMasOrtGFA2003}}).
For $z\in\C$,  we introduce the notation $\tilde D_z=\{\z\in\C \ :\ |z-\zeta|<\rho(\z)\}$. Then we have
\begin{equation}\begin{split}\label{bp11}
&\int_{\C}\left|K(z,\zeta)\right|\,e^{-\phi(\zeta)-\phi(z)}\,dm(\z)
\\ &\int_{|z-\zeta|\ge\max\{\rho(z),\r(\z)\}}\left|K(z,\zeta)\right|\,e^{-\phi(\zeta)-\phi(z)}\,dm(\z)+
\int_{D(z,\r(z))\cup \tilde D_z}\left|K(z,\zeta)\right|\,e^{-\phi(\zeta)-\phi(z)}\,dm(\z)
\\ &= I+II.
\end{split}\end{equation}
Estimating $II$ is the easy part.
Combining \cite[Proposition 2.11]{marzo-ortega} with \cite[Corollary 3]{MarMasOrtGFA2003} we deduce
\begin{equation}\begin{split}\label{bp12}
&II\lesssim \int_{D(z,\r(z))\cup \tilde D_z}\frac{1}{\r(z)\r(\z)}\,dm(\z)
<C<\infty.
\end{split}\end{equation}
\par On the other hand, using  \cite[Theorem 1.1]{marzo-ortega},  \cite[ Lemmas 2 and 4]{{MarMasOrtGFA2003}}, together with \cite[Lemma 2.7]{marzo-ortega}, we get
\begin{equation*}\begin{split}
&I\lesssim \int_{|z-\zeta|\ge\max\{\rho(z),\r(\z)\}}\frac{\,dm(\z)}{\r(z)\r(\z)\exp(d^\ep_\phi(z,\z))}, \quad\text{for some $\ep>0$}
\\ & =\int_{|z-\zeta|\ge\max\{\rho(z),\r(\z)\}}\frac{\r(\z)}{\r(z)}\frac{dm(\z)}{\r^2(\z)\exp(d^\ep_\phi(z,\z))}
\\ & \lesssim \int_{|z-\zeta|\ge\max\{\rho(z),\r(\z)\}}\left(\frac{|z-\z|}{\r(z)}\right)^{1-\d}\frac{dm(\z)}{\r^2(\z)\exp(d^\ep_\phi(z,\z))}
\\ & \lesssim \int_{|z-\zeta|\ge\max\{\rho(z),\r(\z)\}}\frac{d^\g_\phi(z,\z)dm(\z)}{\r^2(\z)\exp(d^\ep_\phi(z,\z))},\quad\text{for some $\d\in (0,1)$}
\\ & \lesssim \int_{|z-\zeta|\ge\max\{\rho(z),\r(\z)\}}\frac{dm(\z)}{\r^2(\z)\exp(d^{\ep'}_\phi(z,\z))},\quad\text{for some $\ep'<\ep$}
\\ & \lesssim \int_{\C}\frac{d\mu(\z)}{\exp(d^{\ep'}_\phi(z,\z))}\le C<\infty,
\end{split}\end{equation*}
which together \eqref{bp11} and \eqref{bp12}
gives   \eqref{bp10}.
\medskip\par{\bf{Case $\mathbf{2}$.}} Assume that $\phi\in\mathcal{S}$.
 We shall follow ideas and the notation from \cite{SeiYouJGA2011}. Again, it is enough to prove \eqref{bp10} that is equivalent to
\begin{equation}\label{2}
\sup_{z\in\C}\int_{\C} |K(z,\zeta)|e^{-1/2(\Psi(|z|^2)+\Psi(|\zeta|^2))}dA(\zeta)\le C_2 <\infty.
\end{equation}
where $C_2$ is an absolute contant.

By \cite[(5.1)]{SeiYouJGA2011},
\begin{equation}\label{1}
\int_{\C}\rho(z,\zeta) |K(z,\zeta)|e^{-1/2(\Psi(|z|^2)+\Psi(|\zeta|^2))}dA(\zeta)\le C_1<\infty.
\end{equation}
where $C_1$ is an absolute constant and the distance $\rho$ is defined by \cite[(1.2)]{SeiYouJGA2011}.

So, by (\ref{1}), for any $\delta>0$
\begin{equation}\begin{split}\label{3}
&\int_{\rho(z,\zeta)\ge \delta} |K(z,\zeta)|e^{-1/2(\Psi(|z|^2)+\Psi(|\zeta|^2))}\,dA(\zeta)
\\ & \le \frac{1}{\delta}\int_{\rho(z,\zeta)\ge \delta} \rho(z,\zeta)|K(z,\zeta)|e^{-1/2(\Psi(|z|^2)+\Psi(|\zeta|^2))}\,dA(\zeta)\le \frac{C_1}{\delta}.
\end{split}\end{equation}
\par On the other hand, by \cite[Lemma $7.2$]{SeiYouJGA2011}, $\{\zeta:\rho(z,\zeta)<\delta \} \subset \{\zeta: |z-\zeta|\le \delta M (\Delta\left(\Psi(|z|^2)\right))^{-1/2}\}$
if $\delta$ is small enough. So, by \cite[Lemma $3.1$, Lemma $3.2$]{SeiYouJGA2011}
 \begin{equation*}\begin{split}
&\int_{\{\rho(z,\zeta)< \delta\}} |K(z,\zeta)|e^{-1/2(\Psi(|z|^2)+\Psi(|\zeta|^2))}\,dA(\zeta)
\\ &\le \int_{\{\zeta: |z-\zeta|\le \delta M (\Delta\left(\Psi(|z|^2)\right))^{-1/2}\}} |K_(z,\zeta)|e^{-1/2(\Psi(|z|^2)+\Psi(|\zeta|^2))}\,dA(\zeta)
\\ & \le \int_{\{\zeta: |z-\zeta|\le \delta M (\Delta\left(\Psi(|z|^2)\right))^{-1/2}\}} \left(K(z,z)\right)^{1/2}\left(K(\zeta,\zeta)\right)^{1/2}e^{-1/2(\Psi(|z|^2)+\Psi(|\zeta|^2))}\,dA(\zeta)
\\ & \lesssim \int_{\{\zeta: |z-\zeta|\le \delta M (\Delta\left(\Psi(|z|^2)\right))^{-1/2}\}} (\Delta\Psi(|z|^2))^{1/2}(\Delta\Psi(|\zeta|^2))^{1/2}\,dA(\zeta)
\\ & \lesssim  C,
\end{split}\end{equation*}
which together \eqref{3} gives (\ref{2}). This finishes the proof.
\end{proof}
Bearing in mind the above result, the dual spaces of weighted Fock spaces  $\mathcal{F}^{\phi}_{p}$ can be described.
\begin{corollary} Asumme that $\phi:\C\to \R^+$ belongs to $\mathcal{D}\cup\mathcal{S}$. For $p>1$, the dual of  $\mathcal{F}^{\phi}_{p}$ is $\mathcal{F}^{\phi}_{q}$,
under the pairing
\begin{equation}\label{pairingf}
\langle f,g \rangle =\int_\C f \bar g \, e^{-2\phi}dm,
\end{equation}
where $1/p+1/q=1$. Moreover, the dual of $\mathcal{F}^{\phi}_{1}$ is $\mathcal{F}^{\phi}_{\infty}$ under the pairing (\ref{pairingf}).
\end{corollary}
The proof is analogous to that of Corollary \ref{dual}, it will be omitted.
\medskip

Finally, we  observe that the analogue of Zeytuncu's unboundedness result for the Bergman projection \cite[Theorem 1.2]{ZeyTams2012} remains true in the complex plane. More precisely, we have
\begin{theorem}\label{th:zeyFock}
If $w$ is a radial weight on $\C$ which satisfies
\par $(i)$\, $w$ is  $C^\infty([0,\infty))$,
\par $(ii)$\,$\lim_{r\to \infty} w^{(n)}(r)=0$ for any $n\in\N$,
\par $(iii)$\, for any $n\in\N$ there exists $a_n\in (0,\infty)$ such that $(-1)^nw^{(n)}$ is non-negative on the interval
$(a_n,\infty)$.
\par Then, for $0<p<\infty$, the Bergman projection  is bounded from  $L^p(\C,w )$ to  $L^p(\C,w )$only for $p=2$.
\end{theorem}
A proof can be obtained by mimicking that of Theorem \cite[Theorem 1.2]{ZeyTams2012}, where  $\int_0^1 r^{2x+1}w(r)dr$ must be replaced by
$\int_0^\infty r^{2x+1}w(r)dr$, so it will be omitted.

\par  With this result in our hands, Proposition \ref{th:2} can be proved analogously to  Proposition \ref{th:1}.


\begin{thebibliography}{99}

\bibitem{AlCo}          A.~Aleman and O.~Constantin,
                        Spectra of integration operators on weighted Bergman spaces,
                        J. Anal. Math. \textbf{109} (2009), 199-–231.

\bibitem{BB}            D.~Bekoll\'e and A.~Bonami,
                        In\'egalit\'es \'a poids pour le noyau de
                        Bergman,
                        (French) C. R. Acad. Sci. Paris Sér. A-B \textbf{286} (1978), no. 18, 775–-778.

\bibitem{CP} O. Constantin and J. A. Pel\'aez,  {\em{Integral operators, embedding theorems and a Littlewood-Paley formula on weighted Fock spaces}}, submitted,
available on http://arxiv.org/abs/1304.7501.

\bibitem{D1} M. Dostanic, {\em{Unboundedness of the Bergman projections on $L^p$ spaces with
exponential weights}}, Proc. Edinb. Math. Soc. \textbf{47} (2004),
111--117.

\bibitem{D3} M. Dostanic, {\em{Boundedness of the Bergman projections on $L^p$ spaces with
radial weights}}, Pub. Inst. Math.  \textbf{86} (2009).
5--20.

\bibitem{Duren1970}     Duren, P.: Theory of $H^p$ Spaces, Academic Press, New York-London 1970.

\bibitem{HKZ}           H.~Hedenmalm, B.~Korenblum and K.~Zhu,
                        Theory of Bergman Spaces, Graduate Texts in Mathematics, Vol. 199,
                        Springer, New York, Berlin, etc. 2000.

\bibitem{JevPac98}  M.~Jevti\'c and .~Pavlovi\'c, On multipliers from $H^p$ to $l^q$ ($0<q<p<1$),
                  Arch. Math. \textbf{56}, 174--180, (1998)

\bibitem{Kriete2003} T.~L.~Kriete, {\em{Laplace Transform Asymptotics, Bergman Kernels and Composition Operators,}}
                                    Operator Theory, Advances and Applications, Vol. {\bf{143}}, 225--272.

\bibitem{MarMasOrtGFA2003} N.~Marco, M.~Massaneda and J.~Ortega-Cerd\`{a}, {\em{Interpolating and sampling sequences for entire functions}},
Geom. Funct. Anal. {\bf{13}}, (2003), 862--914.

\bibitem{marzo-ortega} J. Marzo and J. Ortega-Cerda, {\em{Pointwise estimates for the Bergman kernel of the weighted Fock space}}, J. Geom. Anal. {\bf 19} (2009), 890--910.

\bibitem{PP}        J.~Pau and J.~A. Pel\'{a}ez,{\em{
                    Embedding theorems and integration operators on Bergman spaces with rapidly decreasing weights}},
                    J. Funct. Anal. {\bf{259}} n. 10 (2010), 2727--2756.

\bibitem{Pabook}  M.~Pavlovi\'c, Introduction to function spaces on the
                    Disk, Posebna Izdanja [Special Editions], vol. {\bf{20}}, Matemati\v cki
                    Institut SANU, Beograd, 2004. {\it available online at:}
                    http://f3.tiera.ru/2/M\_Mathematics/MC\_Calculus/MCc\_Complex\%20variable/Pavlovic\%20M.\%20

\bibitem{Padec} M.~Pavlovi\'c, {\em{Analytic functions with decreasing coefficients and Hardy and Bloch spaces}},
 Proc. Edinburgh Math. Soc. Ser. 2 \textbf{56} n. 2 (2013), 623--635.


\bibitem{SeiYouJGA2011}        K.~Seip and E.~H.~Youssfi, {\em{Hankel operators on Fock spaces and
                               related Bergman kernel estimates}}. J. Geom. Anal. {\bf{23}} n. 1 (2013), 170--201.


\bibitem{tung}  J. Tung,  {\em{Fock Spaces}}, PhD dissertation, University of Michigan, 2005.

\bibitem{ZeyTams2012}      Y.~E.~Zeytuncu,  {\em{ $L^p$ regularity of weighted Bergman projections}},
                           Trans. Amer. Math. Soc., \textbf{365} (2013), 2959--2976.

\bibitem{Zhu}              K. Zhu, {\em{Operator Theory in Function Spaces}}, Second Edition,
                           Math. Surveys and Monographs, Vol. 138, American Mathematical
                           Society: Providence, Rhode Island, 2007.


\bibitem{ZhuFock}         K. Zhu, Analysis on Fock Spaces, Springer-Verlag, New York, (2012).

\end{thebibliography}
\end{document}